\documentclass[12pt]{amsart}
\usepackage{amsmath, amssymb, amscd, a4, mathrsfs}

\numberwithin{equation}{section}

\theoremstyle{plain}
\newtheorem{thm}{Theorem}[section]
\newtheorem{lem}[thm]{Lemma}
\newtheorem{prop}[thm]{Proposition}

\newcommand{\thmref}[1]{Theorem~\ref{#1}}
\newcommand{\lemref}[1]{Lemma~\ref{#1}}
\newcommand{\propref}[1]{Proposition~\ref{#1}}

\theoremstyle{definition}
\newtheorem{rmk}[thm]{Remark}

\newtheorem*{hypo*}{Question}
\newtheorem*{hypo2*}{Question'}

\newcommand{\hyporef}[1]{Hypothesis~I}

\newcommand{\mbb}{\mathbb}
\newcommand{\x}{\textbf}
\newcommand{\mf}{\mathbf}
\newcommand{\q}{\quad}
\newcommand{\qq}{\qquad}

\newcommand{\mrm}{\mathrm}
\newcommand{\s}{SL(2, \mbb Z)}
\newcommand{\sr}{SL(2, \mbb R)}

\newcommand{\mr}{\widetilde{SL}(2, \mbb R)}

\newcommand{\g}{\widetilde{\gamma}}

\begin{document}

\title[Jacobi forms and differential operators]{Jacobi Forms and Differential Operators: odd weights}

\author{Soumya Das} 
\address{Department of Mathematics\\
Indian Institute of Science\\
560012, Bangalore, India.}
\email{somu@math.iisc.ernet.in, soumya.u2k@gmail.com}

\author{Ritwik Pal}
\address{Department of Mathematics\\
	Indian Institute of Science\\
560012, Bangalore, India.}
\email{ritwik.1729@gmail.com, ritwik14@math.iisc.ernet.in}

\date{}

\begin{abstract}
We show that it is possible to remove two differential operators from the standard collection of $m$ of them used to embed the space of Jacobi forms of \textit{odd} weight $k$ and index $m$ into several pieces of elliptic modular forms. This complements the previous work of one of the authors in the case of even weights.
\end{abstract}
\maketitle 

\section{Introduction} 
Let $J_{k,m}(N)$ denote the space of Jacobi forms of weight $k$ and index $m$ for the Jacobi group $\Gamma_0(N) \ltimes \mathbb{Z}^2$ of level $N$. It is well known that certain differential operators $D_{\nu}$ ($0 \leq \nu \leq m$ for $k$ even and $1 \leq \nu \leq m-1$ for $k$ odd), first systematically studied in the monograph by Eichler-Zagier (see \cite{ez}), map $J_{k,m}(N)$ injectively into a direct sum of finitely many spaces of elliptic modular forms. See section~\ref{prelim} for a description of these objects. This result has many applications, estimating the dimension of $J_{k,m}(N)$ precisely is one of them. Moreover information about the vanishing of the kernel of $D_0$ when $k=2, m=1$ and $N$ is square--free, has a bearing on the Hashimoto's conjecture on theta series, see \cite{ab1, ab2}.

When the index $m=1$, the question about $\ker D_0$, which is nothing but the restriction map from $J_{k,m}(N)$ to $M_k(N)$, the space of elliptic modular forms of weight $k$ on $\Gamma_0(N)$; defined by $\phi(\tau,z) \mapsto \phi(\tau,0)$, translates into the possibility of removing the differential operator $D_2$ while preserving injectivity. This question is also interesting in its own right, and has received some attention in the recent past, see the works \cite{ab1, ab2, das, rs}, and the introduction there. We only note here that first results along this line of investigation seems to be by J. Kramer, who gave an explicit description of $\ker D_0$ when $m = 1$ and level $N$ a prime, in terms of the vanishing order of cusp forms in a certain subspace of $S_4(N)$ (This is related to the so-called Weierstrass subspaces of $S_k(N)$, see \cite{ab2}).

An unpublished question of S. B\"ocherer, inspired by the case $m=1$ as discussed above, asks for information about this phenomenon when the index in bigger than $1$. Let us recall that below.
\begin{hypo*} \label{hy}
For $0 \leq \nu \leq m$ and $m\geq 1$, the map $i_{\nu}(k,m,N)$ defined by
\begin{align*}  
D_0 \oplus \ldots \widehat{D_{ \nu}} \ldots \oplus D_{2m} \colon  & J_{k,m}(N) \\
& \xrightarrow{i_{\nu}(k,m,N)} M_{k}(N) \oplus \ldots 
\widehat{M_{k+ \nu}(N)} \ldots \oplus M_{k + 2m}(N),
\end{align*}
is \textbf{injective}; where the $\ \widehat{} \ $ signifies that the corresponding term has to be omitted.
\end{hypo*}
Note that it is classical (see also section~\ref{diffsec}) that the above map without any term omitted is an injection, so the above question asks for something stronger. This was answered relatively satisfactorily in \cite{das} when $k$ was even, namely that one can, under certain conditions on $m,k$, remove the last operator $D_{2m}$. 

\textbf{In this paper, we take up the case $k$ odd and henceforth assume this condition.} As expected, there are some subtleties in this case. First of all, it is clear that only differential operators with odd indices matter, and moreover by \cite[p.~37]{ez} it is known that the last operator $D_{2m-1}$ can be removed, since an odd Jacobi form $\phi(\tau,z)$ cannot have a more than $(2m-3)$-fold zero at $z=0$. This shows that the above \textbf{Question} in our context can be phrased as:
\begin{hypo2*} \label{hy2}
For $1 \leq \nu \leq m-1$, $m\geq 3$ with $k$ odd and keeping the above notation, the map
\begin{align*}  
 J_{k,m}(N)  \xrightarrow{ i_{2\nu-1}(k,m,N) } M_{k+1}(N) \oplus \ldots 
\widehat{M_{k+ 2\nu-1}(N) } \ldots \oplus M_{k + 2m-3}(N),
\end{align*}
is \textbf{injective}.
\end{hypo2*}

Note that the right hand side of the map is empty when $m \leq 2$. The aim of this paper is to answer the above \textbf{Question'}. The method is in spirit the same as in \cite{das}, and can be thought of as a sequel to loc. cit., even though one has to be careful about certain subtleties, as the weight is odd. For example, we encounter Wronskians of (congruent) theta functions of weight $3/2$, which seems not to be written down in the literature. Let us state the main theorem of this paper now, which essentially states that we can remove the last operator $D_{2m-3}$. 

\begin{thm}\label{mainthm} Let $k\geq 3$ be an odd integer. Then  
\begin{enumerate}
\item[{(i)}]
$i_{2m-3}(k,m,N)$ is injective for all $N \geq 1$ with $m-k \geq 4$. 

\item[{(ii)}]
$i_{2m-3}(k,m,N)$ is injective for $N$ square-free and $m$ odd with $m-k \geq 2$.

\item[{(iii)}]
$i_{2m-3}(k,m,N)$ is injective for $N=1$ and $m$ odd with $m-k \geq 2$.
\end{enumerate}
\end{thm}

We take this opportunity to note that the proof of \cite[Theorem~1.2~(i)]{das} is not correct as it is; however we stress that this does not affect any other result of the paper, moreover moreover this part of the result was already known before from \cite{rs}.

\textbf{Acknowledgment.}
The first author acknowledges financial support in parts from the UGC Centre for Advanced Studies, DST (India) and IISc, Bangalore during the completion of this work. The second named author thanks NBHM for financial support and IISc, Bangalore where this research was done.

\section{Notation and preliminaries}\label{prelim}
$\Gamma$ will always denote the group $\s$. $M_{k}(\Gamma,\chi)$ (respectively $S_{k}(\Gamma,\chi)$) denotes the space of modular forms (resp. cusp forms) of weight $k$ on $\Gamma$ with character $\chi$. More generally, the space of modular forms of weight $k$ on $\Gamma_{0}(N)$ with a multiplier system $\vartheta$ is denoted by $M_k(N,\vartheta)$. We refer the reader to \cite{rankin} for more details.
\subsection{Jacobi forms}
We have to recall several basic results and notations from \cite{das}. Let $N$, $m$ and $k$ be positive integers. We denote the space of Jacobi forms 
of weight $k$, index $m$, level $N$ by $J_{k,m}(N)$. If $\phi$ is a Jacobi form in $J_{k,m}(N)$, then it has the Fourier expansion $\phi(\tau,z) = \sum_{ n,r\in {\mathbf Z} \colon 4mn \geq r^2  }  c_\phi(n,r) q^n \zeta^r$ and a theta expansion
\begin{equation}\label{thetad}
\phi(\tau,z) = \sum_{\mu \bmod{2m}} h_{m,\mu}(\tau)\theta_{m,\mu}(\tau,z),
\end{equation}
with $\theta_{m,\mu}(\tau,z) =  \sum_{r \in \mathbf{Z}, \, {r\equiv \mu\bmod{2m}}} q^{\frac{r^2}{4m}} \zeta^r$ are the congruent theta series of weight $1/2$ and index $m$; and $h_{m,\mu}(\tau) = \sum_{n \in \mathbf{Z}, \,{n\ge {\mu}^2/4m}}
c_{\phi}(n,\mu)  q^{ (n-\frac {\mu^2}{4m}) }$. We use the standard notation $q := e^{2 \pi i \tau}$ ($\tau \in \mathbb H$) and $\zeta := e^{2 \pi i z}$ ($z \in \mathbb C$).

\subsubsection{Metaplectic Jacobi group}
%

The metaplectic Jacobi group over $\mbb R$ denoted by $\widetilde{J}_{\kappa,m}$ to be $\mr \ltimes \mbb R^2 \cdot \mbb S^1$, where $\mr $ is the metaplectic double cover of $\sr$, $\mbb S^1$ is the circle group. Further, one can defines a certain action of $\widetilde{J}_{\kappa,m}$ on the space of holomorphic functions $\mbb H \times \mathbb{C}$ for $\kappa \in \tfrac{1}{2}  \mathbb{Z}$ and $m\geq 1$:
\[ \phi \mid_{\kappa,m} \widetilde{\xi} := s^m w(\tau)^{-2\kappa}e^{2\pi im (-\frac{c(z+\lambda\tau+\mu)^2}{c\tau +d} +\lambda^2\tau + 2\lambda z + \lambda\mu)}  \phi\left(\frac{a\tau+b}{c\tau+d}, \frac{z+\lambda\tau+\mu}{c\tau+d}\right)\]
where $\widetilde{\xi} = (\g, [\lambda,\mu]s)$, with $\lambda,\mu \in \mbb R$, $s \in \mbb S^1 $. For brevity, we denote the action of the naturally embedded subgroup $\mr$ viz. $\phi \mid_{\kappa,m} (\g, [0,0]1)$ by $\phi \mid_{\kappa,m} \g$. We refer the reader to \cite{sko} for further details.

\subsubsection{Vector valued modular forms} \label{vvmf-sec}
An $n$-tuple $\mf f:= (f_1,\ldots,f_n)^t$ of holomorphic functions on $\mbb H$ is called a vector valued modular form (v.v.m.f.) of weight $\kappa \in \tfrac{1}{2}\mbb Z$ with respect to a representation $\rho \colon \widetilde{\Gamma} \to GL(n,\mbb C)$ if
\[ \mf f \mid_\kappa \g(\tau)= \rho(\g) \mf f(\tau), \q (\tau \in \mbb H, \g \in \widetilde{\Gamma}), \]
and $\mf f$ remains bounded as $\Im(\tau) \to \infty$, where $\mid_\kappa$ operation is defined component-wise on $\mf f$. Let us denote the space of such function by $\widetilde{M}(k,\rho)$.
\subsection{Differential operators} \label{diffsec}
We recall the differential operators $D_\nu$ ($\nu \geq 0$) on $J_{k,m}(N)$ following the exposition in \cite{das}. Namely from the Taylor expansion $\phi(\tau,z) =  \sum_{\nu \geq 0} \chi_{\nu}(\tau) \, z^{\nu}$, of $\phi \in J_{k,m}(N)$ around $z = 0$, one defines 
$D_{ \nu} \phi(\tau) := A_{k,\nu} \xi_{\nu} = A_{k,\nu}~\sum_{0 \leq \mu \leq \frac{\nu}{2} } \, \tfrac{ (- 2 \pi i m)^{\mu} \, (k + \nu - \mu - 2)!   } {(k+2 \nu -2)! \, \mu !}\, \chi_{\nu - 2 \mu}^{(\mu)} (\tau)$, where $A_{k, \nu}$ are certain explicit constants, see e.g., \cite{das}. Then it is known that $ \xi_{\nu} $ is a modular form of weight $k + \nu$ for $\Gamma_0(N)$, and if $ \nu >0$, it is a cusp form.

Furthermore, for all $m,N \geq 1$, the linear map defined by $\oplus_{\nu=0}^m D_{2\nu}$ from $J_{k,m}(N)$ to $M_{k}(N) \oplus_{\nu=1}^m S_{k + 2 \nu}(N)$ is injective.
\section{Setup} \label{setup}
Let $\phi(\tau,z)\in J_{k,m}(N)$ and recall that $k$ is odd. This implies   
$\phi(\tau,z)=-\phi(\tau,-z)$ and hence from (2.4) we can write
\[ \phi(\tau,z)=\underset{\nu \geq 1} \sum\chi_{2\nu-1} z^{2\nu-1}. \]
The theta components $h_{\mu}$ and the theta series $\theta_{m,\mu}$ satisfy, (see \cite{ez}) 
\[ h_{2m-\mu}=- h_{\mu}; \q
\theta_ {m,-\mu}(\tau,z)=\theta_{m,\mu}(\tau,-z) \qquad (\mu \bmod 2m). \]
In particular, $h_{0}=h_{m}=0$. So we get $\phi(\tau,z)$ equals to
\[ \sum_{\mu=1}^{m-1} h_{\mu}(\tau)(\theta_{m,\mu}(\tau,z)-\theta_{m,\mu}(\tau,-z))=\sum_{\mu=1}^{m-1}h_{\mu}(\tau)   \big(\sum_{r\equiv\mu(\bmod{2m})} q^\frac{r^2}{4m}(\zeta^r-\zeta^{-r})  \big).\]
 This implies 
\begin{equation*}
\begin{split}
\chi_{2\nu-1}(\tau)=&\frac{1}{(2\nu-1)!} (\partial_z )^{2\nu-1}\phi(\tau,z)\mid_{z=0}\\
=&\sum_{\mu=1}^{m-1}h_{\mu}(\tau) \sum_{r\equiv\mu(\bmod 2m)} 2 q^\frac{r^2}{4m} (2\pi ir)^{2\nu-1} 
=c_{\nu,m}\sum_{\mu=1}^{m-1}h_{\mu}(\tau)\theta^{{*}{(\nu-1)}}_{m,\mu}(\tau)\text{,}
\end{split}
\end{equation*}
where $c_{\nu,m} = 2\frac{(2\pi i)^{2\nu-1}(4m)^{\nu-1}}{(2\nu-1)!}$; we have put $\theta^{*}_{m,\mu}(\tau)= \displaystyle\sum_{r\equiv\mu(\bmod 2m)}rq^\frac{r^2}{4m} $ and used the notation $\theta^{{*}{(\nu-1)}}_{m,\mu}(\tau)= (\partial_\tau)^{\nu-1}\theta^{*}_{m,\mu}(\tau)$.
\begin{rmk}
It is well-known that the (congruent) theta series $\theta_{m,\mu}^{*}$ ($\mu \bmod 2m$) is a cusp form of weight $\frac{3}{2}$ on $\Gamma(4m)$; see \cite[Prop~2.1]{shimu} for example.
\end{rmk}
Let us put $\chi_{2\nu-1}^{*} = \frac {\chi_{2\nu-1}}{c_{\nu,m}}$. From the above, we get a system of equations connecting $(h_{\mu})_\mu$ and $ (\chi_{2\nu-1}^{*})_\nu$:
 \begin{equation}\label{mat}
\begin{array}{ccccc}
\left[\begin{array}{cccc}
\theta^*_{m,1} & \theta^*_{m,2} & \dots & \theta^*_{m,m-1}\\
\vdots & \vdots & \vdots & \vdots\\
\theta^{*(v-1)}_{m,1} & \theta^{*(v-1)}_{m,2} & \dots & \theta^{*(v-1)}_{m,m-1}\\
\vdots & \vdots & \vdots & \vdots\\
\theta^{*(m-2)}_{m,1} & \theta^{*(m-2)}_{m,2} & \dots & \theta^{*(m-2)}_{m,m-1}\\
\end{array}\right] & . &
\left[\begin{array}{c}
h_1 \\ h_2 \\ \vdots \\ h_\nu \\ \vdots \\ h_{m-1}
\end{array}\right]
& = &
\left[\begin{array}{c}
\chi_{1}^{*} \\\chi_{3}^{*} \\ \vdots \\ \chi_{2\nu-1}^{*} \\ \vdots \\ \chi_{2m-3}^{*}
\end{array}\right]
\end{array}.
\end{equation} 
\begin{lem} \label{F}
$F:=(\theta^{*}_{m,1},...,\theta^{*}_{m,m-1}) \in \widetilde{M}(\frac{3}{2},\rho_{m}) $ for some representation $ \rho_{m}$.
\end{lem} 
\begin{proof}
We know that the tuple  $ (\theta_{m,1}(\tau,z),\theta_{m,2}(\tau,z),...,\theta_{m,2m}(\tau,z))$ is a v.v.m.f. with respect to a representation $\widetilde{\epsilon}$ of $\widetilde{\Gamma}$. Thus there exists $D_{m}(\widetilde{\epsilon}) \in GL_{n} (\mathbb{C})$
such that  we get (see \cite{sko})
\begin{equation*}
(\theta_{m,1}\mid_{\frac{1}{2}}\widetilde{\epsilon}, \theta_{m,2}\mid_{\frac{1}{2}}\widetilde{\epsilon}, \dots , \theta_{m,2m}\mid_{\frac{1}{2}}\widetilde{\epsilon})^t
=D_m(\widetilde{\epsilon})\hspace{0.2cm}
(\theta_{m,1} , \theta_{m,2} , \dots , \theta_{m,2m})^t,
\end{equation*}
where $A^{t}$ is the transpose of $A$. Hence for $\mu, v \bmod{2m}$, there exist $c_{v,\mu} \in \mathbb{C}$ such that
\begin{equation}\label{slt}
 \theta_{m,\mu}(\tau,z)\mid_{\frac{1}{2}}\widetilde{\epsilon}= \sum_ {v\;mod\; 2m}c_{v,\mu} \theta_{m,v}(\tau,z).
\end{equation}
Replacing $z$ by $-z$ in \eqref{slt} and then subtracting from \eqref{slt} we get
\begin{equation}\label{thtil}
\widetilde{\theta}_{m,\mu}(\tau,z)\mid_{\frac{1}{2}}\widetilde{\epsilon}= \sum_ {v\; mod\; 2m}c_{v,\mu}\widetilde{ \theta}_{m,v}(\tau,z),
\end{equation}
where we have put $\widetilde{\theta}_{m,\mu}(\tau,z)=\theta_{m,\mu}(\tau,z)-\theta_{m,\mu}(\tau,-z).$

Let $\widetilde{\epsilon} = (\left( \begin{smallmatrix} a & b \\c & d \end{smallmatrix} \right), w(\tau))\in \widetilde{\Gamma}. $ 
Now recall that 
\[\widetilde\theta_{m,\mu}(\tau,z)\mid_{\frac{1}{2}} \widetilde{\epsilon}= w^{-1}(\tau) e( \frac{-mcz^2}{c\tau+d})\;\widetilde\theta_{m,\mu}(\frac{a\tau+b}{c\tau+d},\frac{z}{c\tau+d} ),\] 
so that by differentiation,
\[ \frac{\partial}{\partial z}(\widetilde\theta_{m,\mu}(\tau,z))\mid_{\frac{1}{2}} \widetilde{\epsilon})\mid_{z=0}=w^{-3}(\tau) \ \theta_{m,\mu}^{*}(\frac{a\tau+b}{c\tau+d})=\theta_{m,\mu}^{*}(\tau)\mid_{\frac{3}{2}} \widetilde{\epsilon}.\]
Since $\frac{\partial}{\partial z} (\widetilde{\theta}_{m,\mu}(\tau,z))\mid_{z=0}=\theta_{m,\mu}^{*}$, by using \eqref{thtil} we get 
\[\theta_{m,\mu}^{*}(\tau)\mid_{\frac{3}{2}} \widetilde{\epsilon}=\displaystyle\sum_{v\; mod\; 2m} c_{v,\mu}\ \theta_{m,v}^{*}(\tau).\]
Note that  $\theta_{m,0}^{*}(\tau)= \theta_{m,m}^{*}(\tau)=0$ and  $\theta_{m,\mu}^{*}(\tau)= -\theta_{m,-\mu}^{*}(\tau)$.
Hence
\[\theta_{m,\mu}^{*}(\tau)\mid_{\frac{3}{2}} \widetilde{\epsilon}= \displaystyle\sum_{v=1}^{m-1}(c_{v,\mu}-c_{2m-v,\mu})\theta_{m,v}^{*}(\tau)\quad\text{ for all }\mu \in \{1,2,...,m-1\}.\]
The lemma then follows from \cite[Prop 3.1]{das}.
\end{proof}                    
Let $\mathcal{W}$ denote the matrix of theta-derivatives appearing in equation \eqref{mat}. It is crucial for us to find a formula for $\det\mathcal{W}$ (i.e. the Wronskian of $(\theta_{m,1}^{*},...,\theta_{m,m-1}^{*})$). 

At this point, following \cite{das}, we introduce the differential operator $\mathbb{D}_\kappa$, sometimes referred to as `modular derivative'. Namely for $f$ holomorphic on $\mathbb{H}$ and $ E_{2} = 1-24\sum_{n=1}^{\infty}\sigma_{1}(n)q^{n} $, the weight 2 quasimodular Eisentein series, we define for $\kappa \in \frac{1}{2} \mathbb{Z}$ the operator
 \[\mathbb{D}_{\kappa}  f := q\frac{d}{dq}-\frac{\kappa}{12} E_{2}f.\]
 The operator $\mathbb{D}_{\kappa}$ enjoys the following equivariance property under $\widetilde{\Gamma}$:
 \[\mathbb{D}_{\kappa}  (f\mid_{\kappa} \widetilde{\gamma})=(\mathbb{D}_{\kappa}  f)\mid_{\kappa+2}\widetilde{\gamma}\quad\text{for all}\; \;\widetilde{\gamma} \in 
\widetilde{\Gamma},\]
which shows that the operator  $\mathbb{D}_{\kappa}$ can be iterated; and setting
\[\mathbb{D}^{n}=\mathbb{D}_{\kappa}^ {n}=\mathbb{D}_{\kappa+2n-2}\circ...\mathbb{D}_{\kappa+2}\circ\mathbb{D}_\kappa,\]
we have the equivariance property (see \cite[p. 358]{das})
\begin{equation}\label{moduD}
\mathbb{D}_{\kappa}^{n} (f\mid_{\kappa}\widetilde{\gamma})=(\mathbb{D}_{\kappa}^{n}  f)\mid_{\kappa+2n}\widetilde{\gamma} \quad\text{for all} \; \widetilde{\gamma} \in 
 \widetilde{\Gamma}. 
 \end{equation}
The following lemma is an immediate consequence of \eqref{moduD}. We omit the proof. 

\begin{lem}\label{dj}
With $F$ defined as in \lemref{F}, for any $ j \in \mathbb{N}$ and $F\in \widetilde{M}(\frac{3}{2},\rho_{m}) $, $\mathbb{D}^{j} F \in \widetilde{M}(\frac{3}{2}+2j,\rho_{m})$.
\end{lem} 

Now let us define the `modular Wronskian' $W(F)$ of $F$ given by
 \[ W(F) := \det(F,\mathbb{D}F,\mathbb{D}^{2}F,...,\mathbb{D}^{m-2}F).\]
Using the fact $q\frac{d}{dq}= \frac{1}{2\pi i}\frac{d}{d\tau}$, it is straightforward to prove the following by using column operations on $\mathcal{W}$ that \[W(F)=q^{\frac{(m-1)(m-2)}{2}}\det\left(\frac{d\theta_{m,\mu}^{*(j-1)}}{dq^{(j-1)}}\right)_{\mu,j}=c\cdot\det\left(\frac{d\theta_{m,\mu}^{*(j-1)}}{d\tau^{(j-1)}}\right)_{\mu,j}=c\cdot\det\mathcal{W}\]
where $\mu,j$ vary in $[1,m-1]$ and $c \in \mathbb{C}^*=\mbb{C} - \{0\}$. The following lemma is straightforward and we omit the proof (see \cite{mason} for the case of the modular group).

\begin{lem}\label{detf}
If $F_{i}=(f_{i1},f_{i2},...,f_{ip}) \in\widetilde{M} (k_{i},\rho)$ then $ \det(f_{ij})_{1\leq i,j \leq p} \in \widetilde{M}(k,\det\rho)$, where $k=k_{1}+k_{2}+...+k_{p}$.
\end{lem}

If $f$ is a non-zero holomorphic function on $\mathbb{H}$, we define $\mrm{ord}_\infty f$ to be the power of $q$ in the first non-zero term of the $q$-expansion of $f$.
\begin{lem}
$W^{2}(F)= c_{1}^{2}\cdot q^{\frac{(m-1)(2m-1)}{12}}+ \ldots \text{higher powers}$.
\end{lem}
Here $c_{1}=\text{c}\cdot(m-1)!\cdot V(\frac{1^2}{4m},\frac{2^2}{4m},...,\frac{(m-1)^2}{4m})$ and $ V(a_{1},a_{2},...,a_{n})$ denotes the Vandermonde determinant of quantities ${a_{1},a_{2},...a_{n}}$.

\begin{proof}
As $W(F)=c\cdot \det\mathcal{W}$, we get $\mrm{ord}_\infty W(F)= \mrm{ord}_\infty \det\mathcal{W}$. Clearly the first non-zero coefficient in the $q-$expansion of $\det \mathcal{W}$ is the determinant of the first non-zero coefficients of all entries of $\mathcal{W}$. From this we get the first non-zero coefficient of $\det\mathcal{W}$ to be $(m-1)!\cdot V(\frac{1^2}{4m},\frac{2^2}{4m}, \ldots ,\frac{(m-1)^2}{4m})$. Now the order at $\infty$ of each element in the $\mu$-th column of $\mathcal{W}$ is $\frac{\mu^{2}}{4m}$. Therefore one has $\mrm{ord}_\infty \det\mathcal{W}=\sum_{\mu={1}}^{m-1}\frac{\mu^{2}}{4m}=\mrm{ord}_\infty W(F)$, and equivalently $\mrm{ord}_\infty W^{2}(F)= 2\cdot\sum_{\mu=1}^{m-1}\frac{\mu^{2}}{4m}=\frac{(m-1)(2m-1)}{12}$.
\end{proof}	
\subsection{Automorphy of the Wronskian} \label{xi} 
In this subsection we prove that $W(F)$ is a certain integral power of the Dedekind-$\eta$ function. Whereas this property was known in the classical case by a result of J. Kramer \cite{kr2}, and was invoked in \cite{das}, we have to work it out, and the method should apply to the classical theta functions as well.

\noindent ---\textbf{A character of $\Gamma$}. We now introduce a particular character of $\Gamma$, which will  essentially allow us to replace $\widetilde{\Gamma}$ by $\Gamma$. Consider the following commutative diagram:
\begin{equation} \label{xi}
\begin{CD}
\widetilde{\Gamma}    @>\pi>>  \Gamma \\
@V(\det\rho_{m})^{2} VV       @VV \xi V \\
\mathbb{C}^*    @>id>>       \mathbb{C}^* ,
\end{CD}
\end{equation}
where $\pi$ is the projection map and we define $\xi:\Gamma\rightarrow \mathbb{C}^{*}$ so that the above diagram commutes. In fact let us put $\xi(\gamma) := (\det\rho_{m})^{2}(\pi^{-1} (\gamma))$, so that we just have to check that $\xi$ is well-defined. This is true because $(\det\rho_{m})^{2}$ is clearly trivial on $\ker(\pi)=\langle \widetilde{S}^{4} \rangle $, where $\widetilde{S}^{4}=(\left( \begin{smallmatrix} 1 & 0 \\0 & 1 \end{smallmatrix} \right),-1)$. To proceed further we need more information about the character $\xi$ of $\Gamma$. This is obtained from the following lemmas.
\begin{lem}
With $\rho_{m}$ as in \lemref{F}, one has,	\[\rho_{m}(\widetilde{T})= \mathrm{diag} \left(e(\frac{1^{2}}{4m}), e(\frac{2^{2}}{4m}),...,e(\frac{(m-1)^{2}}{4m}) \right).\]
\end{lem} 

\begin{proof}
It is clear from the definition of $\theta_{m,\mu}^{*}$ that, $\theta_{m,\mu}^{*}(\tau+1)=e(\frac{\mu^{2}}{4m})\cdot\theta_{m,\mu}^{*}(\tau)$.
Now $\rho_{m}(\widetilde{T})F(\tau)=F\mid_{\frac{3}{2}}\widetilde{T}=F(\tau+1)=\mathrm{diag}(e(\frac{1^{2}}{4m}), e(\frac{2^{2}}{4m}),...,e(\frac{(m-1)^{2}}{4m}))F(\tau)$. The $q-$expansions of the $\theta_{m,\mu}^{*}$'s clearly show that they are linearly independent. Thus the lemma follows.
\end{proof}

Let us now recall that the quotient of $\Gamma$ by its commutator subgroup $[\Gamma, \Gamma]$ is cyclic of order $12$ and is generated by (the image of) $T$. The group of characters $\widehat{\Gamma / [\Gamma, \Gamma]}$ is also cyclic and $\widehat{\Gamma / [\Gamma, \Gamma]}=\widehat{\Gamma}$. We denote the canonical generator of $\widehat{\Gamma}$ by $\delta$, uniquely characterised by $\delta(T)=e(\frac{1}{12})$ (see e.g.,           \cite[p.~380]{mason}).

\begin{lem}{\label{xidelta}}
With the above notation, $ \xi =\delta^{(m-1)(2m-1)}$.	
\end{lem}

\begin{proof}
From Lemma (3.6) we get
$\det\rho_{m}(\widetilde{T})=e(\sum_{\mu=1}^{m-1}\frac{\mu^{2}}{4m})$. This implies that 
$\xi(T)=e(\frac{(m-1)(2m-1)}{12}) = \delta^{(m-1)(2m-1)}(T)$. Therefore, $ \xi =\delta^{(m-1)(2m-1)}$.
\end{proof}

\begin{lem}
 $W^{2}(F)\in M_{(m-1)(2m-1)}(\Gamma,\delta^{(m-1)(2m-1)})$.	
\end{lem}

\begin{proof}
From \lemref{dj} and \lemref{detf} we conclude that $W(F)\in M((m-1)(m-\frac{1}{2}), \det\rho_{m})$. From the definition of $\xi$ in \eqref{xi} it then follows that $ W^{2}(F)\in M_{(m-1)(2m-1)}(\Gamma,\xi)$. Finally \lemref{xidelta} completes the proof.
\end{proof}

\begin{prop} \label{propimp}
$W(F)=c_{2}\cdot\eta^{(m-1)(2m-1)}$, where $\eta=q^{\frac{1}{24}}\cdot\prod_{n=1}^{\infty}(1-q^{n})$ and $c_2 \in \mathbb{C}^{*}$.
\end{prop}

\begin{proof}
We know that $\eta^{2(m-1)(2m-1)}\in S_{(m-1)(2m-1)}(\Gamma,\delta^{(m-1)(2m-1)})$ and  
\[ \eta^{2(m-1)(2m-1)}=q^{\frac{(m-1)(2m-1)}{12}}+\ldots \text{higher powers} \quad (\text{see \cite[Thm 3.7]{mason}) }.\]
Putting $G:=W^2(F) / \eta^{2(m-1)(2m-1)}$ we get $G\in M_{0}(\Gamma,1)=\mathbb{C}$. As $W(F)$ is non-zero, $G\in \mathbb{C}^{*}$. The proposition follows.
\end{proof}

From the description of differential operators $D_{\nu}$ in terms of Taylor coefficients $\chi_{\nu}$ of $\phi$ we can easily infer that,
\[\phi \in \cap_{\nu=1}^{m-2} \ker D_{2\nu-1} \iff \chi_{2\nu-1}=0\text{ for }\nu=1, 2,...,m-2.\] Thus if $\phi$ satisfy the above condition and $(\frac{\omega_1}{\det\mathcal{W}},\frac{\omega_2}{\det\mathcal{W}},...,\frac{\omega_{m-1}}{\det\mathcal{W}})$ being the last column of the matrix $\mathcal{W}^{-1}$, we get using \eqref{mat}
\begin{equation}\label{omega}
h_{\mu}=c_3\cdot\frac{\omega_\mu\cdot\chi_{2m-3}}{\eta^{(m-1)(2m-1)}},
\end{equation} for $\mu=1,2,...,m-1$ and some non-zero constant $c_3$.

\begin{prop} \label{om}
The tuple $\Omega_{m-1}:=(\omega_1,\omega_2,...,\omega_{m-1}) \in \widetilde{M}(\frac{(m-2)(2m-3)}{2},\alpha_{m})$ for some representation $\alpha_{m}$.
\end{prop}
The proof of this proposition is rather similar to the proof of \cite[Prop~3.2]{das} with some appropriate and obvious changes, and is omitted. The weight of $\Omega_{m-1}$ can be calculated from \eqref{omega}.

\section{Proof of \thmref{mainthm}}
As in \cite{das}, the following theorem will be useful to prove \thmref{mainthm} in certain cases.
\begin{thm} \label{vanish}
The space  $S_k(N)^\eta := \{f \in S_k(N) \colon f \ \text{divisible by } \ \eta^{2k-2}   \}$ is zero for  
	
	$(i)$~$N=1$ if $12 \nmid k$, or for
	
	$(ii)$~all square-free $N$, provided  $k \equiv 4,10  ~\bmod 12$. \label{backbone}
\end{thm}
Here `divisible' means divisible in the ring of holomorphic modular forms (with multiplier). Part $(i)$ is not stated in \cite{ab2}, but follows from the valence formula. 

For the proof of \thmref{mainthm}~$(ii)$ and~$(iii)$, we need information about the orders at $\infty$ of the functions $\omega_\nu$ in \eqref{omega}. The next lemma is devoted to that.

\begin{lem}\label{ord}
$\mrm{ord}_\infty\omega_\nu=\sum_{\mu=1}^{m-1}\frac{\mu^2}{4m}-\frac{\nu^2}{4m}$ \qq \qq $(\nu=1,2,...,m-1)$.
\end{lem} 

\begin{proof}
	$\omega_\nu$ is the determinant of $(m-1),\nu$-th minor of $\mathcal{W}$. Looking at the first non-zero co-efficients of entries of $\mathcal{W}$ we see that the leading non-zero coefficient of $\omega_\nu$ equals $c_{4}\cdot\frac{(m-1)!}{\nu}\cdot V(\frac{1^2}{4m},..,\widehat{\frac{\nu^2}{4m}},...,\frac{(m-1)^2}{4m})$ for some $c_4\in \mathbb{C}^*$. Now each element in the $\mu$-th column
	of the $\mathcal{W}$ has $\mrm{ord}_{\infty}$ equal to $\frac{\mu^2}{4m}$.	Thus the lemma follows.
\end{proof}	

\noindent\textbf{Proof of \thmref{mainthm}, cont'd.} We now go back to \eqref{omega}, where we multiply both sides by a non-zero $ H \in M_s(\Gamma_0(N)) \backslash \{0\}$, where $s \geq 0$ is even to be specified later. This gives
\begin{equation}\label{4.1}
\frac{\chi_{2m-3}\cdot H}{\eta^\beta}=\frac{h_{\mu}\cdot H}{\omega_{\mu}\cdot \eta^{\beta-\lambda}}.
\end{equation}
where we have put 
\begin{align} \label{betadef}
\beta=2(k+2m+s-4), \qq
 \lambda=(m-1)(2m-1).\end{align}
Our aim is to prove $\chi_{2m-3}=0$. To this end, let us define a new function $\psi$ by
\[\psi=\frac{h_{\mu}\cdot H}{\omega_{\mu} \cdot \eta^{\beta-\lambda}}.\]
From \eqref{4.1} it is clear that $\psi$ does not depend on $\mu$ and since from \lemref{4.1} we get $\omega_\mu$ is never zero, $\psi$ is well defined. It is clear from \eqref{4.1} that if $\psi$ is bounded at all cusps of $\Gamma_{0}(N)$, then $\psi \in M_1(N,\vartheta)$ for some multiplier $\vartheta$. \textit{We want to find sufficient conditions on $m,k$ which would imply that $\psi=0$, i.e., $\chi_{2m-3}=0$.} With that aim, it is helpful here to introduce another function for $r\in \mathbb{Z}$, defined by
\begin{equation} \label{4.2}
\psi_{r}=\eta^{-r} \cdot \psi =\frac{h_{\mu}\cdot H}{\omega_{\mu} \cdot\eta^{\beta-\lambda+r}}.
\end{equation}
From the previous discussion and \eqref{4.2} we see that $\psi_{r}\in M_{1-\frac{r}{2}}(N,\varepsilon_r)$ for some multiplier $\varepsilon_r$ if we can show that it is bounded at all cusps. In other words it is enough to show that $\psi_{r}\mid_{1-\frac{r}{2}}^{\varepsilon_r}\gamma$ is bounded for any $\gamma\in \Gamma$ and where we recall that
\[\psi_{r}\mid_{1-\frac{r}{2}}^{\varepsilon_r}\gamma=\varepsilon_r(\gamma)^{-1} j(\gamma,\tau)^{-(1-\frac{r}{2})} \psi_{r}(\gamma\tau).\]
Let us recall from \propref{om} that  $\Omega_{m-1}:=(\omega_1,\omega_2,...,\omega_{m-1}) $ belongs to $\widetilde{M}(\frac{(m-2)(2m-3)}{2},\alpha_{m})$. We fix $\widetilde{\gamma_0}=(\gamma_{0}, \omega_{0}(\tau)) \in \widetilde{\Gamma} $ and we calculate:
\begin{align}
& \varepsilon_r(\gamma_0)^{-1} w_0(\tau)^{-(2-r)} \psi_{r}( \gamma_0 \tau) \cdot    \alpha_m(\g_0) \left( \Omega_{m-1}(\tau) / \eta^{\lambda-\beta-r}(\tau) \right)   \label{00} \\
&= \varepsilon'(\gamma_0) j(\gamma_0,\tau)^{-k -s + 1/2} \psi_{r}( \gamma_0 \tau) (\Omega_{m-1} / \eta^{\lambda-\beta-r})(\gamma_0 \tau) \nonumber \\
&= \varepsilon''(\gamma_0) j(\gamma_0,\tau)^{-k- s + 1/2} (h_1 H,h_2 H,\ldots, h_{m-1} H)(\gamma_0 \tau) \nonumber \qq (\text{using} \, \eqref{4.2}) \\
&= \varepsilon'''(\gamma_0) \left[ (h_1H,h_2H,\ldots, h_{m-1}H)\mid^{\varepsilon''}_{k +s - 1/2} \gamma_0 \right].\label{bdd}
\end{align}
In the above, $\varepsilon'(\gamma_0)$, $\varepsilon''(\gamma_0)$, $\varepsilon'''(\gamma_0)$ are certain complex numbers which arise from the appropriate multiplier systems are roots of unity. Since $h_\mu,H$ are modular forms, \eqref{bdd} remains bounded as $ \Im(\tau) \to \infty$.

In order to prove that $\psi_{r}\mid_{1-\frac{r}{2}}^{\varepsilon_r}\gamma$ is bounded, it is sufficient that as $\Im(\tau)\to \infty$, atleast one of the components of $\Omega_{m-1}(\tau)/\eta^{\lambda-\beta-r}$ is non-zero. In other words we need some $\nu\in \{1,2,...,m-1\}$ such that $\mrm{ord}_{\infty}\eta^{\lambda-\beta-r}\geq \mrm{ord}_{\infty}\omega_{\nu}$. Now concernig the possibility of such a $\nu$, putting $\nu=m-l$ leads us to
  \[2(m-k) \geq 12l- \frac{6l^2}{m}+2s+r-8.\] 
From the above, $m-k$ will be minimal, i.e., we get the largest range of values of $m-k$ when $l=1$ i.e. $\nu=m-1$. If we further consider the possibility of  $ \mrm{ord}_{\infty}\eta^{\lambda-\beta-r}= \mrm{ord}_{\infty}\omega_{m-1}$, it would imply $\lambda-\beta-r=\frac{(m-2)(m-1)(2m-3)}{m}$, which is impossible for $m>3$ since the right hand side is not an integer. Thus we are led to consider
\begin{align} \label{cond}
\mrm{ord}_{\infty}\omega_{m-2}> \mrm{ord}_{\infty}\eta^{\lambda-\beta-r} > \mrm{ord}_{\infty}\omega_{m-1},
\end{align}
which along with \lemref{ord} implies
\begin{align} \label{infty}
\Omega_{m-1}(\tau)/\eta^{\lambda-\beta-r} \to (0,0,...,0,\infty) \q \text{as} \q \Im(\tau) \to \infty.
\end{align}
We also note that for $m>3$, \eqref{cond} is equivalent to the condition
\[(m-2)+\frac{3}{m}>k+s+\frac{r}{2}>(m-8)+\frac{12}{m},\]
i.e.
\begin{align} \label{eqinq}
 s+\frac{r}{2}+8-\frac{12}{m} > m-k > 2+s+\frac{r}{2}-\frac{3}{m}.
\end{align}
In our application, $r$ would be even. So assuming this condition, the second inequality in \eqref{eqinq} above implies that the largest range of $m-k$ happens, and the first inequality is also satisfied, if we choose
\begin{align} \label{mks}
m-k=2+s+\frac{r}{2} \quad \text{for} \q m>3.
\end{align}
Henceforth we assume \eqref{mks}. Now \eqref{00}, \eqref{bdd}, \eqref{cond} and \eqref{infty}  imply that $\psi_{r}\mid_{1-\frac{r}{2}}^{\varepsilon_{r}} \gamma$ vanishes as $\Im(\tau)\to\infty$, since the last column of $\alpha_m(\widetilde{\gamma_0})$ is not identically zero because $\alpha_m$ is a homomorphism from $\widetilde{\Gamma}$ to $GL_{m-1}(\mathbb{C})$. The upshot of the foregoing discussion is that 
\begin{align} \label{bang}
\psi_{r}\in M_{1-\frac{r}{2}}(N,\varepsilon_r) \q \text{if} \q 2+s+\frac{r}{2}=m-k.
\end{align}

\noindent$\bullet$\, \textbf{Part (i) of \thmref{mainthm}}: In this case $N\geq1$ is arbitrary, and we choose $r>2$ even and $s=0$. This means $1- r/2<0$ and so $\psi_{r}\in M_{1-\frac{r}{2}}(N,\varepsilon_r) = \{0\}$ (see e.g. \cite[Theorem~4.2.1]{rankin}) when $m-k\geq4$. Thus under this condition $\chi_{2m-3}=0$.

\noindent$\bullet$\, \textbf{Part (ii) of \thmref{mainthm}}: In this case $N$ is square-free. In \eqref{mks} we choose $r=0$, so that $s=m-k-2\geq 0$. To have a non-zero $H$ in \eqref{4.2} we need to have $s\in 2\mathbb{Z}$ and hence $m$ to be odd. By \eqref{betadef} and \eqref{bang} we know that $\eta^\beta$ divides $\chi_{2m-3} \cdot H$ with $\beta$ suited for an application of part~$(i)$ of \thmref{vanish} provided
\[k+2m-3+s\equiv 4, 10\;\pmod {12} \quad
\text{ i.e., }\quad k+2m+s\equiv 1\; \pmod{6}.\]
From the choice of $s=m-k-2\geq 0$ we get, \[k+2m+s=3m-2\equiv1\;\pmod{6}.\] Hence when $N$ is square-free, $m$ is odd and $m-k\geq 2$, one has $\chi_{2m-3}=0$.

\noindent$\bullet$\, \textbf{Part (iii) of \thmref{mainthm}}: Here $N=1$. We argue exactly as in the previous case. We set $r=0$, so that $s=m-k-2\geq 0$. As before we can apply part~$(ii)$ of \thmref{vanish} to conclude $\chi_{2m-3} \cdot H=0$ provided
\[2m+k+s-3\not\equiv 0\pmod{12} \quad
\text{i.e., } \quad 2m+k+s\not\equiv 3\pmod{12}.\]
Putting the value of $s$, we see that this condition is always satisfied. Therefore when $N=1$, $m$ is  odd and $m-k\geq2$, one has $\chi_{2m-3}=0$.


\begin{thebibliography}{22}

\bibitem{ab1} T. Arakawa, S. B\"ocherer: {\em A Note on the Restriction Map for Jacobi Forms},  Abh. Math. Sem. Univ. Hamburg \x{69} (1999), 309--317. 

\bibitem{ab2} T. Arakawa, S. B\"ocherer: {\em Vanishing of certain spaces of elliptic modular forms and some applications},  J. reine angew Math. \x{559} (2003), 25--51.

\bibitem{das} S. Das and B. Ramakrishnan: {\em Jacobi forms and differential operators}, J. Number Theory 149 (2015), 351-367.

\bibitem{ez} M.\ Eichler and D.\ Zagier: {\em The Theory of Jacobi Forms}. Progress in Mathematics, \x{Vol. 55},  Boston-Basel-Stuttgart: Birkh\"auser, 1985.

\bibitem{kr2}
J. Kramer:  {\em Jacobiformen und Thetareihen}, Manuscripta Math. \x{54} (1986), 279--322.

\bibitem{marks} C. Marks: {\em Classification of vector-valued modular forms of dimension less than six}, arXiv:1003.4111v1. 


\bibitem{mason} G. Mason: {\em Vector-valued modular forms and linear differential operators}, Int. J. Number Theory \x{3} (2007), no. 3, 377--390.

\bibitem{mi} T. Miyake, {\em Modular forms}, Translated from the Japanese by Yoshitaka Maeda. Springer--Verlag, Berlin, (1989), x+335 pp.

\bibitem{rs} B. Ramakrishnan and Karam Deo Shankhadhar: {\em On the restriction map for Jacobi forms}, Abh. Math. Semin. Univ. Hambg. {\bf 83} (2013), 163--174.

\bibitem{rankin} R. A. Rankin: {\em Modular forms and functions}. Cambridge University Press, Cambridge-New York-Melbourne, 1977. xiii+384 pp. 


\bibitem{shimu} G. Shimura: {\em On modular forms of half integral weight}, Annals of Mathematics  Vol. 97, No. 3 (May, 1973), pp. 440-481

\bibitem{sko} N. P. Skoruppa: {\em \"Uber den Zusammenhang zwischen Jacobiformen und Modulformen halbganzen Gewichts}. Dissertation, Rheinische Friedrich-Wilhelms-Universit\"at, Bonn, 1984. Bonner Mathematische Schriften [Bonn Mathematical Publications], 159. Universit\"at Bonn, Mathematisches Institut, Bonn, 1985. vii+163 pp.

\end{thebibliography}
\end{document}